\documentclass[12pt, reqno]{amsart}
\usepackage{amssymb,verbatim,amscd,amsmath,graphicx,enumerate}
\usepackage{graphicx}
\usepackage{caption}
\usepackage{subcaption}
\usepackage{pdfsync}
\usepackage{fullpage}
\usepackage{color}
\usepackage{marginnote}
\usepackage{tikz}

\usetikzlibrary{patterns}
\usetikzlibrary{calc}
\usetikzlibrary{matrix}
\usepgflibrary{arrows}
\usetikzlibrary{arrows}
\usetikzlibrary{decorations.pathreplacing}
\usetikzlibrary{decorations.pathmorphing}
\usepackage[pagebackref]{hyperref}
\numberwithin{equation}{section}
\newtheorem{theorem}{Theorem}[section]
\newtheorem{lemma}[theorem]{Lemma}

\newtheorem{corollary}[theorem]{Corollary}

\theoremstyle{definition}
\newtheorem{definition}[theorem]{Definition}

\newtheorem{def-prop}[theorem]{Definition-Proposition}
\newtheorem{remark}[theorem]{Remark}
\newtheorem{example}[theorem]{Example}

\newtheorem*{acknowledgement}{Acknowledgement}

\newtheorem{setup}[theorem]{Setup}
\newtheorem*{Mysketch}{Sketch of proof} % \newtheorem establishes the object heading
  {\pushQED{\qed}\begin{Mysketch}}
  {\popQED\end{Mysketch}}

\DeclareMathOperator{\reg}{reg}
\DeclareMathOperator{\stab}{stab}

\DeclareMathOperator{\Spec}{Spec}
\DeclareMathOperator{\Proj}{Proj}

\newcommand{\CC}{{\mathbb C}}

\newcommand{\PP}{{\mathbb P}}
\newcommand{\ZZ}{{\mathbb Z}}
\newcommand{\NN}{{\mathbb N}}

\def\mm{{\mathfrak m}}

\def\pp{{\mathfrak p}}
\def\qq{{\mathfrak q}}

\def\O{{\mathcal O}}
\def\I{{\mathcal I}}

\def\L{{\mathcal L}}
\def\M{{\mathcal M}}

\def\E{{\mathcal E}}

\def\R{{\mathcal R}}
\def\S{{\mathcal S}}

\def\F{{\mathcal F}}

\def\1{{\bf 1}}
\def\0{{\bf 0}}

\def\ix{{\overline{X}}}
\def\tx{{\widetilde{X}}}

\begin{document}

\title{Fiber invariants of projective morphisms and regularity of powers of ideals}

\author{Sankhaneel Bisui}
\address{Tulane University \\ Department of Mathematics \\
6823 St. Charles Ave. \\ New Orleans, LA 70118, USA}
\email{sbisui@tulane.edu}
%\urladdr{http://www.math.tulane.edu/$\sim$tai/}

\author{Huy T\`ai H\`a}
\address{Tulane University \\ Department of Mathematics \\
6823 St. Charles Ave. \\ New Orleans, LA 70118, USA}
\email{tha@tulane.edu}
%\urladdr{???}

\author{Abu Chackalamannil Thomas}
\address{Tulane University \\ Department of Mathematics \\
6823 St. Charles Ave. \\ New Orleans, LA 70118, USA}
\email{athoma17@tulane.edu}
%\urladdr{???}


\begin{abstract}
We introduce an invariant, associated to a coherent sheaf over a projective morphism of schemes, which controls when sheaf cohomology can be passed through the given morphism. We then use this invariant to estimate the stability indexes of the regularity and $a^*$-invariant of powers of homogeneous ideals. Specifically, for an equigenerated homogeneous ideal $I$ in a standard graded algebra over a Noetherian ring, we give bounds for the smallest values of power $q$ starting from which $a^*(I^q)$ and $\reg(I^q)$ become linear functions.
\end{abstract}

\maketitle

%%%%%%%%%%%%%%%%%%%%%%%%%%%
\begin{center}
{\it In honor of Professor L\^e V\u{a}n Thi\^em's centenary}
\end{center}

\section{Introduction} \label{sec.intro}

A celebrated result, proven independently by Kodiyalam \cite{K} and Cutkosky, Herzog and Trung \cite{CHT}, states that if $I$ is a homogeneous ideal in a standard graded algebra over a field then the regularity of $I^q$ is asymptotically a linear function; that is, there exist constants $d$ and $b$ such that $\reg(I^q) = dq+b$ for all $q \gg 0$. This result was extended to standard graded algebras over a Noetherian ring by Trung and Wang \cite{TW} (see also \cite{BCH, W} for the $G$-graded situation, where $G$ is any abelian group). A similar statement for the closely related $a^*$-invariant and, more generally, all the $a$-invariants was established by the author in \cite{Ha} and by Chardin in \cite{Ch1}. The constant $d$ was implicitly described in \cite{K} and made more precise in \cite{TW}. It has been an important problem since then to understand the constant $b$ and the minimum power starting from which $\reg(I^q)$ becomes a linear function (cf. \cite{B, Ch1, Ch2, Ch3, EH, EU, Ha, Tim}).  Computing these invariants for special classes of ideals have also been the subject of many recent works (cf. \cite{AB, BBH, BHT, Borna, BCV, Gu, HTT, HangT, JNS, JS1, JS2, Lu, R, SF, SFY}).

Much of attention was paid toward the particular case when $I$ is equigenerated. Let $A$ be a standard graded algebra over a Noetherian ring $A_0$, and let $I = (f_0, \dots, f_m) \subseteq A$ be a homogeneous ideal generated by $(m+1)$ forms of degree $d > 0$. Let $X = \Proj A \subseteq \PP^n_{A_0}$, let $\ix$ be (the closure of) the image of the rational map $\varphi: X \dashrightarrow \PP^m_{A_0}$ defined by $[f_0:\dots:f_m]$, and let $\tx$ be the blowup of $X$ centered at $I$. Then, $\tx$ can be naturally identified with (the closure of) the graph $\Gamma$ of $\varphi$ inside the bi-projective space $\PP^n_{A_0} \times \PP^m_{A_0}$. Under this identification, the morphism $\pi: \tx \rightarrow X$, induced by the projection map $\PP^n_{A_0} \times \PP^m_{A_0} \rightarrow \PP^n_{A_0}$, coincides with the blowing-up morphism of $X$ centered at $I$. We have the following diagram
$$\begin{array}{rcccl} & & \tx \simeq \Gamma & \subseteq & \PP^n_{A_0} \times \PP^m_{A_0} \\
\pi & \swarrow & & \searrow & \phi \\
X & & \stackrel{\varphi}{-\joinrel\dasharrow} & & \ix \subseteq \PP^m_{A_0}\end{array}$$

It turns out that local invariants associated to the projection map $\phi$ in fact govern the constants in the asymptotic linear forms of $a^*(I^q)$ and $\reg(I^q)$, for $q \gg 0$. More specifically, in a series of work \cite{Ch1, EH, Ha}, it was proven that if $a^*_\phi$ and $\reg_\phi$ are the maximum $a^*$-invariant and regularity of fibers of $\phi$ (see Section \ref{sec.ainv} for precise definitions) then for all $q \gg 0$, we have
$$a^*(I^q) = dq+a^*_\phi \text{ and } \reg(I^q) = dq + \reg_\phi.$$

The stability indexes of $I$, namely,
\begin{align*}
\stab_{\reg}(I) & = \min\{q_r ~\big|~ \reg(I^q) = dq+\reg_\phi \ \forall q \ge q_r\} \\
\stab_a(I) & = \min\{q_a ~\big|~ a^*(I^q) = dq+a^*_\phi \ \forall \ q \ge q_a\},
\end{align*}
have also been investigated in \cite{B, Ch3, EU} for $A_+$-primary ideals, and in \cite{Ch2} for equigenerated ideals.
More precisely, let $\R = A[It]$ be the Rees algebra of $I$. Then, $\R$ carries a natural bi-graded structure given by $\R = \bigoplus_{p,q \in \ZZ}\R_{(p,q)}$, where $\R_{(p,q)} = [I^q]_{p+dq}t^q$. We can also view $\R$ as a $\ZZ$-graded ring with $\R = \bigoplus_{p \in \ZZ}\R_{(p,*)} = \bigoplus_{q \in \ZZ}\R_{(*,q)}$, where
$$\R_{(p,*)} = \bigoplus_{q' \in \ZZ} \R_{(p,q')} \text{ and } \R_{(*,q)} = \bigoplus_{p' \in \ZZ}\R_{(p',q)}.$$
Particularly, $\R_{(p,*)}$ is a graded algebra over $\R_{(0,*)} = A_0[I t]$ and $\R_{(*,q)}$ is a graded algebra over $\R_{(*,0)}=A$.
The following results have been obtained.
\begin{enumerate}
\item \cite[Proposition 6.7]{Ch2} Suppose that for any $\pp \in \Proj A$, $\reg (A_\pp[I_\pp t]) = 0$. Then
\begin{align}
a^*(I^q) \le dq + a^*_\phi \text{ for all } q > \max_{p > a^*_\phi}\{a^*(\R_{(p,*)})\}.\label{eq.main}
\end{align}
\item \cite[Theorem 1.5]{Ch3} and \cite[Theorem 1.1]{EU} (see also \cite{B}) Suppose that $I$ is a homogeneous $A_+$-primary ideal in a standard graded polynomial ring $A$ over a field. Then, for all $q > a^*\big(\R_{(a^*_\phi+1,*)}\big)$, we have
\begin{align}
\reg(I^q) = dq + \reg_\phi. \label{eq.maincor}
\end{align}
\end{enumerate}

The goal of this paper is to extend the bound (\ref{eq.main}) in two directions. On one hand, we shall show that the upper bound (\ref{eq.main}) holds without any restriction on $I$. Instead of imposing the condition $q > \max_{p > a^*_\phi}\{a^*(\R_{(p,*)})\}$, we only need $q > a^*(\R_{(a^*_\phi+1,*)})$. The tradeoff is that we shall also need $q > a^*_\pi$. In many special cases of interest, for example, when $I$ satisfies the condition of \cite[Proposition 6.7]{Ch2} or when the Rees algebra of $I$ is Cohen-Macaulay, this requirement $q > a^*_\pi$ is redundant (i.e., $a^*_\pi \le 0$). On the other hand, we shall achieve a lower bound for $a^*(I^q)$ and, as a consequence, utilize local and global invariants associated to both projection maps $\pi$ and $\phi$ to give an estimate for the stability index $\stab_a(I)$ of $I$. We shall now state our main theorem, leaving unexplained notations until later.

\noindent{\bf Theorem \ref{thm.main}.}
Let $A$ be a standard graded algebra over $A_0$ and let $I \subseteq A$ be a homogeneous ideal generated in degree $d > 0$. Let $\pi: \widetilde{X} \rightarrow X$ be the blowup of $X$ along the ideal sheaf of $I$ and let $\phi: \widetilde{X} \rightarrow \overline{X}$ be the natural projection onto its second coordinate.
\begin{enumerate}
\item For all $q > \max \left\{a^*_\pi, a^*\left(\R_{(a^*_\phi+1,*)}\right)\right\}$, we have
	$$a^*(I^q) \le dq + a^*_\phi.$$
\item For all $q \ge \max\left\{a^*_\pi+1, \reg\left(\widetilde{\R_{(a^*_\phi,*)}}\right), \reg^\phi_*(\O_\tx(a^*_\phi,0))\right\},$ we have
$$a^*(I^q) \ge dq+ a^*_\phi.$$
\end{enumerate}
In particular,
$$\stab_a(I) \le \max \left\{a^*_\pi+1, a^*\left(\R_{(a^*_\phi+1,*)}\right)+1, \reg\left(\widetilde{\R_{(a^*_\phi,*)}}\right), \reg^\phi_*(\O_\tx(a^*_\phi,0))\right\}.$$

As a consequence of Theorem \ref{thm.main}, if the $a^*$-invariant \emph{defect sequence} $\{a^*(I^q) - dq\}_{q \in \NN}$ of $I$ is a non-increasing sequence, e.g., when $A$ is polynomial ring over a field and $I$ is a homogeneous $A_+$-primary ideal, then we have $a^*(I^q) = dq + a^*_\phi$ for all $q > \max\big\{a^*_\pi, a^*\big(\R_{(a^*_\phi+1, *)}\big)\big\}$; see Corollary \ref{cor.defect}. Particularly, we recover the equigenerated version of (\ref{eq.maincor}).

To prove Theorem \ref{thm.main}, we investigate the vanishing of sheaf cohomology groups of $\widetilde{I}^q(p+dq)$, for $p,q \in \ZZ$, where $\widetilde{I}$ is the ideal sheaf of $I$ on $X$. Our method for this investigation is to lift these cohomology groups through $\pi$ to sheaf cohomology of $\O_\tx(p,q)$ over $\tx$, then push those cohomology groups through $\phi$, and finally examine when the resulting sheaf cohomology groups on $\ix$ vanish.

To implement this method, we begin by studying the question of when sheaf cohomology can be pushed forward and/or pulled backward through a projective morphism of schemes. To this end, associating to a projective morphism $\pi: Y \rightarrow X$ and a coherent sheaf $\F$ on $Y$, we introduce an invariant $a^*_\pi(\F)$ that governs the higher direct images of $\F$ through $\pi$. Particularly, in Theorem \ref{pro.higherimage}, we establish basic vanishing and nonvanishing properties of sheaf cohomology groups of $\F$ that are controlled by this invariant.
Such an invariant $a^*_\pi(\F)$ was first introduced in \cite{HT} in the study of arithmetic Macaulayfication of projective schemes, and later used in \cite{Ch1, Ch2, Ch3, Ha} in studying the $a^*$-invariant and regularity of powers of ideals.

\begin{acknowledgement} The authors would like to thank Marc Chardin and Ngo Viet Trung for many stimulating discussions on the regularity of powers of ideals over the years. The second named author is partially supported by Louisiana Board of Regents (grant \#LEQSF(2017-19)-ENH-TR-25).
\end{acknowledgement}

%%%%%%%%%%%%%%%%%%%%%%%%%%%%%%%%%%%%%%%%%%%%%%%%%%%%%

\section{Preliminaries} \label{sec.prel}

In this section, we collect important notation and terminology used in the paper. For unexplained basics, we refer the interested reader to the standard texts \cite{BS, Hartshorne}. Throughout the paper, all rings and schemes are assumed to be Noetherian.

\subsection*{Regularity and $a$-invariants.} Regularity was initially defined by Mumford \cite{Mum} for coherent sheaves on projective schemes. This notion was then generalized to finitely generated graded modules over graded algebras (cf. \cite{BS}).

\begin{definition}
Let $R$ be a finitely generated graded algebra over a ring $A = R_0$. Let $X = \Proj R$ and let $\F$ be a coherent sheaf on $X$. For an integer $r \in \NN$, we say that $\F$ is \emph{$r$-regular} if $H^i(X, \F(r-i)) = 0$ for all $i > 0$. The \emph{regularity} of $\F$ is defined to be
$$\reg(\F) = \min \{r \in \ZZ ~\big|~ \F \text{ is } r\text{-regular}\}.$$
\end{definition}

Since $H^i(X, \F) = 0$ for all $i \gg 0$, and $H^i(X,\F(n)) = 0$ for all $i > 0$ and $n \gg 0$ (cf. \cite[Theorems II.2.7 and II.5.3]{Hartshorne}), $\reg (\F)$ is a well-defined and finite invariant.

\begin{definition}
Let $R$ be a finitely generated graded algebra over a ring $A = R_0$. Set $R_+ = \bigoplus_{n > 0}R_n$. Let $M$ be a finitely generated graded $R$-module. For $i \ge 0$, the $i$-th \emph{$a$-invariant} of $M$ is defined to be
$$a^i(M) = \left\{\begin{array}{ll} \sup \ \left\{n \in \ZZ ~\big|~ \left[H^i_{R_+}(M)\right]_n \not= 0\right\} & \text{ if } H^i_{R_+}(M) \not= 0 \\
-\infty & \text{ otherwise.}\end{array}\right.$$
The \emph{$a^*$-invariant} and \emph{regularity} of $M$ are defined as follows:
$$a^*(M) = \max_{i \ge 0} \{a^i(M)\} \text{ and } \reg(M) = \max_{i\ge 0} \{a^i(M)+i\}.$$
\end{definition}

It is well known that $a^i(M) < \infty$ for all $i \ge 0$, and that $a^i(M) = -\infty$ for all $i \gg 0$ (cf. \cite[Theorem 15.1.5]{BS} and \cite[Theorem 2.1]{Ch1}).
Thus, $a^*(M)$ and $\reg(M)$ are well-defined and finite invariants.

\begin{remark} Let $X = \Proj R$ be a projective scheme over a ring $A$, and let $F$ be a finitely generated graded $R$-module. Let $\F$ be the coherent sheaf on $X$ associated to $F$. Then, the Serre-Grothendieck correspondence gives
	$$\reg(F) \ge \reg(\F).$$
\end{remark}

\subsection*{Coherent sheaves and Serre-Grothendieck correspondence.} Let $R$ be a standard graded algebra over a ring $A$ and let $R_+ = \bigoplus_{n > 0}R_n$ be its irrelevant ideal. Let $X = \Proj R$ and let $\widetilde{F}$ be a coherent sheaf associated to an $R$-module $F$ on $X$. The Serre-Grothendieck correspondence gives a short exact sequence
$$0 \rightarrow H^0_{R_+}(F) \rightarrow F \rightarrow \bigoplus_{n \in \ZZ} H^0(X, \widetilde{F}(n)) \rightarrow H^1_{R_+}(F) \rightarrow 0,$$
and isomorphisms, for $i \ge 1$,
$$H^{i+1}_{R_+}(F) \simeq \bigoplus_{n \in \ZZ} H^i(X, \widetilde{F}(n)).$$

\begin{definition}\label{def.global}
Let $\pi: Y \rightarrow X$ be a projective morphism of projective schemes. Let $\F$ be a coherent sheaf on $Y$. Define
$$\reg^\pi_j(\F) = \reg(R^j \pi_* \F) \text{ and } \reg^\pi_*(\F) = \max_{j \ge 0}\{\reg (R^j \pi_* \F)\}.$$
\end{definition}

Note that $\pi_*\F$ and $R^j \pi_* \F$ are coherent sheaves on $X$ and that $R^j \pi_* \F = 0$ for all $j \gg 0$ (cf. \cite[Proposition III.8.5 and Theorem III.8.8]{Hartshorne}). Thus, $\reg^\pi_j(\F)$ and $\reg^\pi_*(\F)$ are well-defined and finite invariants. The following lemma is a well known consequence of the definition of regularity.

\begin{lemma} \label{lem.global}
Let $\pi: Y \rightarrow X$ be a projective morphism of projective schemes and let $\F$ be a coherent sheaf on $Y$. Let $\O_X(1)$ be a very ample invertible sheaf on $X$. Then for all $j \ge 0$ and $q \ge \reg^\pi_j(\F)$, $R^j \pi_* \F \otimes_{\O_X} \O_X(q)$ is generated by global sections.
\end{lemma}

\begin{proof} The assertion follows from \cite[Lecture 14]{Mum}.
\end{proof}

\subsection*{Blowing-up morphisms and bi-projective schemes.} Let $X$ be a scheme and let $\S$ be a sheaf of graded $\O_X$-algebras. The Proj of $\S$ is defined as follows (cf. \cite{Hartshorne}). For each open subset $U = \Spec A$ of $X$, let $\S_U = \Gamma(U, \S\big|_U)$. Then $\S_U$ is a graded $A$-algebra. Consider the projective scheme $\Proj \S_U$ and its natural morphism $\pi_U: \Proj \S_U \rightarrow \Spec A$. These schemes and morphisms glue together to give a scheme $\Proj \S$, a natural morphism $\pi: \Proj \S \rightarrow X$, and an induced \emph{twisting sheaf} $\O(1)$ on $\Proj \S$ (note that $\O(1)$ is not necessarily very ample on $\Proj \S$).

\begin{definition} Let $X$ be a scheme and let $\I$ be a coherent sheaf of ideals on $X$. Consider the sheaf of graded $\O_X$-algebras $\R = \bigoplus_{n \ge 0} \I^n$. The \emph{blowup of $X$ along $\I$} is defined to be $\Proj \R$ together with the natural blowing-up morphism $\pi: \Proj \R \rightarrow X$.
\end{definition}

When $X$ is an affine scheme, the blowup of $X$ along an ideal sheaf is closely related to the familiar notion of Rees algebras, which we shall now recall.

\begin{definition} Let $A$ be a Noetherian ring and let $I \subseteq A$ be an ideal. The \emph{Rees algebra} of $I$ is defined to be the graded $A$-algebra
$$A[It] = A \oplus It \oplus I^2t^2 \oplus \dots \subseteq A[t].$$
\end{definition}

By construction, if $X = \Spec A$ is an affine scheme and $I \subseteq A$ is an ideal then the blowup of $X$ along the ideal sheaf of $I$ is given by $\pi: \Proj A[It] \rightarrow \Spec A$.

\begin{lemma}[\protect{\cite[Lemma 1.1]{HT}}] \label{lem.aRees}
Let $A$ be a Noetherian ring and let $\R = A[It]$ be the Rees algebra of an ideal $I \subseteq A$. If $\R$ is a Cohen-Macaulay ring then $a^*(\R) = -1$.
\end{lemma}

For the purpose of this paper, we are particularly interested in the blowup of a projective scheme. Let $A$ be a standard graded algebra over a ring $A_0$ and let $I \subseteq A$ be a homogeneous ideal. The Rees algebra $A[It]$ of $I$, in this case, also carries a natural bi-graded structure, namely $A[It] = \bigoplus_{p,q \in \ZZ} (I^q)_p t^q$. We can define the bi-projective scheme $\text{Bi-Proj } A[It]$ of $A[It]$ with respect to this bi-graded structure.

\begin{lemma} \label{lem.biproj}
Let $A$ be a standard graded algebra over a ring $A_0$ and let $X = \Proj A$. Let $I \subseteq A$ be a homogeneous ideal and let $\I$ be the ideal sheaf on $X$ associated to $I$. Then, the blowup of $X$ along $\I$ is naturally identified with the bi-projective scheme $\text{Bi-Proj } A[It]$.
\end{lemma}

\begin{proof} As before, let $\R = \bigoplus_{n \ge 0}\I^n$ and let $\tx = \Proj \R$ be the blowup of $X$ along $\I$.

Consider any point $\pp \in X$, and let $A_{(\pp)}$ denote the homogeneous localization of $A$ at the homogeneous prime ideal $\pp \subseteq A$. Then, $U = \Spec \O_{X,\pp} = \Spec A_{(\pp)}$ is an open neighborhood of $\pp$ in $X$. Moreover, we have
$\Gamma(U, \R\big|_U) = A_{(\pp)}[I_{(\pp)}t].$ Thus, $\tx$ is obtained by gluing the affine blowups $\Proj A_{(\pp)}[I_{(\pp)}t] \rightarrow \Spec A_{(\pp)}$ for $\pp \in X$. The twisting sheaf on $\Proj A_{(\pp)}[I_{(\pp)}t]$ is given by $I_{(\pp)}\O_{\Proj A_{(\pp)}[I_{(\pp)}t]}$. Since, by definition, $\I_\pp = I_{(\pp)}$ for all $\pp \in X$, these sheaves glue together to give $\I\O_\tx = \O_\tx(1)$, which agrees with the twisting sheaf of $\tx$.

This construction gives $\tx$ a bi-projective structure, which is the same as the bi-projective structure of $\text{Bi-Proj } A[It]$. Hence, we have $\tx \simeq \text{Bi-Proj } A[It]$.
\end{proof}

%%%%%%%%%%%%%%%%%%%%%%%%%%

\section{Fiber $a$-invariants} \label{sec.ainv}

The aim of this section is to introduce the $a^*$-invariant and regularity associated to a coherent sheaf over a projective morphism of schemes, and to show that these invariants govern when sheaf cohomology groups can be passed through the given projective morphism. For simplicity of notation, throughout this section, unless stated otherwise, we shall be in the following setup.

\begin{setup} \label{setup1}
Let $X$ be a scheme and let $\R$ be a sheaf of finitely generated graded $\O_X$-algebras. Let $Y = \Proj \R$ and let $\O_Y(1)$ be the induced twisting sheaf on $Y$. Let $\pi: Y \rightarrow X$ be the resulting projective morphism, and let $\F$ be a coherent sheaf on $Y$.
\end{setup}

For an open affine subset $U = \Spec A \subseteq X$, let
$$\R_U = \Gamma(U, \R\big|_U) \text{ and } \F_U = \Gamma(U, \F\big|_U).$$
Then, by construction, $\R_U$ is a finitely generated graded $A$-algebra and $\F_U$ is a finitely generated graded $\R_U$-module. Thus, the $a$-invariants and regularity of $\F_U$ are defined as in Section \ref{sec.prel}. We shall extend these notions to define the fiber $a$-invariants and regularity of $\F$ over the morphism $\pi$ as follows. Our construction is slightly more general than that given in \cite{Ch1}, where similar invariants, introduced in \cite{Ha, HT}, were generalized to coherent sheaves over $\PP(\E)$ for a locally free coherent sheaf $\E$ of finite rank on $X$.

\begin{definition} Assume that we are in Setup \ref{setup1}.
\begin{enumerate}
\item Let $\pp \in X$ and $i \ge 0$. The \emph{$i$-th local $a$-invariant} of $\F$ at $\pp$ is defined to be
$$a^i_\pp(\F) = a^i(\F_{\Spec \O_{X,\pp}}) \text{ and } \reg_\pp(\F) = \reg(\F_{\Spec \O_{X,\pp}}).$$
\item The \emph{fiber $a$-invariants} and \emph{regularity} of $\F$ over $X$ are defined to be
\begin{align*}
& a^i_\pi(\F) = \sup_{\pp \in X} \{a^i_\pp(\F)\}, \ a^*_\pi(\F) = \max_{i \ge 0}\{a^i_\pi(\F)\}, \\
& \reg_\pi(\F) = \sup_{\pp \in X} \{\reg_\pp(\F)\} = \max_{i \ge 0}\{a^i_\pi(\F) + i\}.
\end{align*}
\item Define $r_\pi(\F) = \min\{r ~\big|~ a^*_\pi(\F) = a^r_\pi(\F)\}$.
\end{enumerate}
\end{definition}

When there is no confusion, we shall often write $a^*_\pi$ and $\reg_\pi$ for $a^*_\pi(\O_Y)$ and $\reg_\pi(\O_Y)$.

\begin{lemma} \label{lem.integer}
For any $\pp \in X$ and $i \ge 0$, we have $a^i_\pp(\F) < \infty$ and $a^i_\pi(\F) < \infty$. In particular, $a^*_\pi(\F) < \infty$ and $\reg_\pi(\F) < \infty$.
\end{lemma}

\begin{proof} The second statement follows from the definition of $a^*_\pi(\F)$ and $\reg_\pi(\F)$. Also, by definition, $a^i_\pp(\F) < \infty$ for any $\pp \in X$ and $i \ge 0$. We shall complete the proof by showing that $a^i_\pi(\F) < \infty$.

Since $X$ is a Noetherian scheme, $X$ is quasi-compact and, thus, can be covered by a finite number of affine schemes. Thus, without loss of generality, we may assume that $X = \Spec A$ is an affine scheme. In this case, $\R$ is a graded $A$-algebra, and $\F$ is a finitely generated graded $\R$-module. In particular, we have $a^i(\F) < \infty$.

Observe that
$\R_{\Spec \O_{X,\pp}} = \R \otimes_A A_\pp \text{ and } \F_{\Spec \O_{X,\pp}} = \F \otimes_A A_\pp.$
For simplicity of notation, set $\R_\pp = \R \otimes_A A_\pp$ and $\F_\pp = \F\otimes_A A_\pp$.
Since local cohomology commutes with localization, we have
$$H^i_{\R_{\pp+}}(\F_\pp) = H^i_{\R_+}(\F) \otimes_A A_\pp.$$
This implies that for all $\pp \in X$, $a^i_\pp(\F) \le a^i(\F)$. Hence, $a^i_\pi(F) \le a^i(\F) < \infty$.
\end{proof}

\begin{remark}
Inspired by the proof of Lemma \ref{lem.integer}, for any point $\pp \in X$, we shall set $\R_\pp = \R_{\Spec \O_{X,\pp}}$ and $\F_\pp = \F_{\Spec \O_{X,\pp}}$. Then $\R_\pp$ is a graded $\O_{X,\pp}$-algebra and $\F_\pp$ is a finitely generated graded $\R_\pp$-module.
	
Note also that for each $n \in \ZZ$, the degree $n$ piece $(\F_\pp)_n$ of $\F_\pp$ is an $\O_{X,\pp}$-module. Let $\F_n$ be the sheaf on $X$ obtained by gluing sheaves $\widetilde{(\F_\pp)_n}$ on $\Spec \O_{X,\pp}$ for $\pp \in X$. %({\bf Why is $\F_n$ well defined?})
\end{remark}

A particular situation that we are interested in is when $\pi: Y \rightarrow X = \Proj A$ is the blowing-up morphism along the ideal sheaf of a homogeneous ideal $I \subseteq A$, and $\F$ is the structure sheaf of $Y$. In this case, if $\F$ has good local properties, for example, $\F$ is locally Cohen-Macaulay over $X$, or if $I$ is nice enough, for instance, $I$ satisfies the condition of \cite[Proposition 6.7]{Ch2}, then we can effectively bound the invariants $a^*_\pi$ and $\reg_\pi$.

\begin{definition} The sheaf $\F$ is said to be \emph{locally Cohen-Macaulay over $X$} if for all $\pp \in X$, $\F_{\Spec \O_{X,\pp}}$ is a Cohen-Macaulay $\R_{\Spec \O_{X,\pp}}$-module.
\end{definition}

\begin{lemma}[\protect{\cite[Lemma 1.2]{HT}}] \label{lem.lCM}
Let $A$ be a standard graded algebra over a ring $A_0$ and let $X = \Proj A$. Let $I \subseteq A$ be a homogeneous ideal and let $\pi: \tx \rightarrow X$ be the blowing-up morphism of $X$ along the ideal sheaf of $I$. Then, $a^*_\pi \ge -1$ and the equality holds if $\O_\tx$ is locally Cohen-Macaulay over $X$.
\end{lemma}

\begin{example} \label{ex.points}
Let $A = k[x_0, \dots, x_n]$ and let $I \subseteq R$ be the defining ideal of a fat point scheme in $\PP^n$. Let $\pi: \widetilde{\PP^n} \rightarrow \PP^n$ be the blowing up of $\PP^n$ along the ideal sheaf of $I$. In this case, $I$ has the form $I = \bigcap_{i=1}^s \pp_i^{m_i}$, where $\pp_i$ is the defining ideal of a closed point in $\PP^n$ and $m_i \in \NN$.

Let $\pp \in \PP^n$ be any point. Observe that if $\pp \not= \pp_i \ \forall \ i$ then $A_{(\pp)}[I_{(\pp)}t] = A_{(\pp)}[t]$ is a polynomial ring over $A_{(\pp)}$, and so it is Cohen-Macaulay. If, on the other hand, $\pp = \pp_i$ for some $i$, then $A_{(\pp)}[I_{(\pp)}t] = A_{(\pp)}[\pp_{(\pp)}^{m_i}t]$ is the Veronese subalgebra of $A_{(\pp)}[\pp_{(\pp)}t]$, where the later is Cohen-Macaulay since $\pp$ is a complete intersection. Thus, $A_{(\pp)}[I_{(\pp)}t]$ is Cohen-Macaulay. Hence, in this example, $\O_{\widetilde{\PP^n}}$ is locally Cohen-Macaulay over $\PP^n$.
\end{example}

\begin{example} \label{ex.determinant}
Let $A = k[x_{ij} ~\big|~ 1 \le i \le r, \ 1 \le j \le s]$ and let $I \subseteq A$ be the ideal generated by $t \times t$ minors of the generic matrix $M = \left(x_{ij}\right)_{1 \le i \le r, 1 \le j \le s}$ for some $1 \le t \le \min\{r,s\}$. Let $X = \Proj A$ and let $\pi: \tx \rightarrow X$ be the blowing up of $X$ along the ideal sheaf of $I$.

By \cite[Theorem 3.5]{EHu} and \cite[Theorem 3.3]{Bruns}, the Rees algebra $\R = A[It]$ is a Cohen-Macaulay ring. Let $\mm_\R$ be the maximal homogeneous ideal of $\R$. Then, $H^i_{\mm_\R}(\R) = 0$ for all $i < \dim \R$.
This implies that for any $\pp \in X$, $H^i_{\mm_\R}(\R) \otimes_A A_{(\pp)} = 0$ for all $i < \dim \R.$
Since local cohomology commutes with localization, it follows that $A_{(\pp)}[I_{(\pp)}t]$ is a Cohen-Macaulay ring for any $\pp \in X$. Hence, in this example, $\O_\tx$ is also locally Cohen-Macaulay over $X$.
\end{example}

Note that by the same arguments as in Example \ref{ex.determinant}, if the Rees algebra $A[It]$ is a Cohen-Macaulay ring and $\pi: \tx \rightarrow X$ is the blowup centered at $I$ then $\O_\tx$ is locally Cohen-Macaulay over $X$.

\begin{lemma} \label{lem.Prop64}
Let $A$ be a standard graded algebra over a ring $A_0$. Let $I \subseteq A$ be a homogeneous ideal such that for any $\pp \in \Proj A = X$, $\reg (A_\pp[I_\pp t]) = 0$. Let $\pi: \tx \rightarrow X$ be the blowup of $X$ along the ideal sheaf of $I$. Then $a^*_\pi \le \reg_\pi \le 0$.
\end{lemma}

\begin{proof} Let $R = A[It]$ and let $\R = \bigoplus_{n \ge 0}\I^n$, where $\I$ is the ideal sheaf of $I$ on $X$. For any point $\pp \in X$, we have that $\O_{X,\pp} = A_{(\pp)}$ is the homogeneous localization of $A$ at $\pp$. Thus,
$$\R_{\Spec \O_{X,\pp}} = R \otimes_A A_{(\pp)} = A_\pp[I_\pp t] \otimes_{A_\pp} A_{(\pp)}.$$
The assertion follows from the definition and $a^*$-invariant and regularity, the fact that local cohomology commutes with localization.
\end{proof}

\begin{example} \label{ex.localreg0}
Let $A$ be a standard graded algebra over a ring $A_0$. Let $I$ be a homogeneous ideal such that for any $\pp \in \Proj A$, $I_\pp$ is generated by a $d$-sequence (see \cite{Huneke} for more details on $d$-sequences). This condition is satisfied, for instance, if $I$ is a locally complete intersection. Then, by \cite[Corollary 5.2]{Tr2}, for every $\pp \in \Proj A$, we have $\reg (A_\pp[I_\pp t]) = 0$. Hence, we are in the setting of Lemma \ref{lem.Prop64} and, in this case, we have $a^*_\pi \le \reg_\pi \le 0$.
\end{example}

The following theorem is the main result of this section, that establishes important properties of fiber $a^*$-invariant and regularity. This result plays the key role in our study of stability indexes carried out in the next section.

\begin{theorem}\label{pro.higherimage}
Assume Setup \ref{setup1} and suppose further that $X$ is a projective scheme. Let $\O_X(1)$ be a very ample invertible sheaf on $X$, and set $a = a^*_\pi(\F)$.
\begin{enumerate}
\item[(1)] For all $n > a$, we have
	$$\pi_* \F(n) = \F_n \text{ and } R^j\pi_* \F(n) = 0 \ \forall \ j > 0.$$	
\item[(2)] If $r_\pi(\F) \le 1$ then for all $q \ge \max\{\reg(\F_{a}), \reg^\pi_0(\F(a))\}$,
$$H^0(Y, \F(a) \otimes_{\O_Y} \pi^* \O_X(q)) \not= H^0(X, \F_{a}(q)).$$
\item[(3)] If $r_\pi(\F) \ge 2$ then for all $q \ge \reg^\pi_*(\F(a))$,
$$H^{r_\pi(\F)-1}(Y, \F(a) \otimes_{\O_Y} \pi^* \O_X(q)) \not= 0.$$
\end{enumerate}
\end{theorem}

\begin{proof} (1) For any point $\pp \in X$, let $Y_\pp = Y \times_X \Spec \O_{X,\pp}$. Then, $\O_{Y_\pp} = \O_Y \times_{\O_X} \O_{X,\pp}$ and $\F(n)\big|_{Y_\pp} = \F(n) \otimes_{\O_Y} \O_{Y_\pp} = \F(n) \otimes_{\O_X} \O_{X,\pp} = (\F \otimes_{\O_X} \O_{X,\pp})(n).$
For any $j \ge 0$, we have
\begin{align}
R^j \pi_* \F(n) \big|_{\Spec \O_{X,\pp}} = H^j(Y_\pp, \F(n)\big|_{Y_\pp})\!\widetilde{\quad \quad}. \label{eq.10}
\end{align}
The Serre-Grothendieck correspondence gives a short exact sequence
$$0 \rightarrow H^0_{\R_{\pp+}}(\F_\pp) \rightarrow \F_\pp \rightarrow \bigoplus_{n \in \ZZ} H^0(Y_\pp, \F(n)\big|_{Y_{\pp}}) \rightarrow H^i_{\R_{\pp+}}(\F_\pp) \rightarrow 0$$
and isomorphisms, for $i \ge 1$,
$$H^{i+1}_{\R_{\pp+}}(\F_\pp) \simeq \bigoplus_{n \in \ZZ} H^i(Y_\pp, \F(n)\big|_{Y_\pp}).$$
This, together with the definition of $a^*_\pi(\F)$, implies that for $n > a^*_\pi(\F)$,
\begin{align*}
H^0(Y_\pp, \F(n)\big|_{Y_\pp}) & = (\F_\pp)_n, \\
R^j \pi_* \F(n)\big|_{\Spec \O_{X,\pp}} & = 0 \text{ for } j > 0.
\end{align*}
These hold for any $\pp \in X$. Hence,
\begin{align*}
\pi_* \F(n)  = \F_n \text{ and } R^j \pi_* \F(n)  = 0 \ \forall \  j > 0.
\end{align*}

(2) Set $r = r_\pi(\F)$. By definition,
\begin{align}
\left\{\begin{array}{ll} \big[H^i_{\R_{\pp+}}(\F_{\pp})\big]_a  = 0 & \text{ for } i < r \text{ and any } \pp \in X, \\
\big[H^r_{\R_{\qq+}}(\F_{\qq})\big]_a  \not=  0 & \text{ for some } \qq \in X.\end{array}\right. \label{eq.nonvanish}
\end{align}

Since $r \le 1$, together with the Serre-Grothendieck correspondence, this gives
$$H^0(Y_\qq, \F(a)\big|_{Y_\qq}) \not= (\F_\qq)_a.$$
Therefore, by (\ref{eq.10}), we have
$\pi_* \F(a) \not= \F_a.$
Twisting by $\O_X(q)$, we get $\pi_* \F(a) \otimes_{\O_X} \O_X(q) \not= \F_a(q)$ for all $q \in \ZZ$.

By Lemma \ref{lem.global}, for all $q \ge \max\{\reg(\F_{a}), \reg^\pi_0(\F(a))\}$, both $\pi_* \F(a) \otimes_{\O_X} \O_X(q)$ and $\F_a(q)$ are generated by global sections. Thus, we obtain
$$H^0(X, \pi_* \F(a) \otimes_{\O_X} \O_X(q)) \not= H^0(X, \F_a(q))  \ \forall \ q \ge \max\{\reg(\F_{a}), \reg^\pi_0(\F(a))\}.$$
Moreover, by the projection formula, we have $\pi_* (\F(a) \otimes_{\O_Y} \pi^* \O_X(q)) = \pi_* \F(a) \otimes_{\O_X} \O_X(q)$. This implies that $H^0(Y,\F(a) \otimes_{\O_Y} \pi^* \O_X(q)) = H^0(X, \pi_* \F(a) \otimes_{\O_X} \O_X(q))$. Hence, the assertion follows.

(3) It follows from (\ref{eq.nonvanish}), the Serre-Grothendieck correspondence, and (\ref{eq.10}) that
\begin{align*}
R^j \pi_* \F(a)  = 0 \text{ for } 0 < j < r-1 \text{ and } R^{r-1} \pi_* \F(a)  \not= 0. \
\end{align*}
Thus, by the projection formula, we have
\begin{align}
\left\{ \begin{array}{lcl} R^j \pi_* (\F(a) \otimes_{\O_Y} \pi^* \O_X(q)) & = & 0 \text{ for } 0 < j < r-1, \\ R^{r-1} \pi_* (\F(a) \otimes_{\O_Y} \pi^* \O_X(q)) & \not=& 0. \end{array}\right.\label{eq.30}
\end{align}

By Lemma \ref{lem.global}, for $q \ge \reg^\pi_{r-1}(\F(a))$, $R^{r-1} \pi_* (\F(a) \otimes_{\O_Y} \pi^* \O_X(q)) = R^{r-1} \pi_* \F(a) \otimes_{\O_X} \O_X(q)$ is generated by global sections. Therefore, by (\ref{eq.30}), we deduce that for $q \ge \reg^\pi_{r-1}(\F(a))$,
\begin{align}
H^0(X, R^{r-1} \pi_* (\F(a) \otimes_{\O_Y} \pi^* \O_X(q))) \not= 0. \label{eq.not0}
\end{align}
Moreover, for $q \ge \reg^\pi_0 (\F(a))$, we have $H^i(X, \pi_* \F(a)) = 0$ for all $i > 0$. Hence, by considering the Leray spectral sequence
$$E^{i,j}_2 = H^i(X, R^j \pi_* (\F(a) \otimes_{\O_Y} \pi^* \O_X(q))) \Rightarrow H^{i+j}(Y, \F(a) \otimes_{\O_Y} \pi^* \O_X(q)),$$
(\ref{eq.30}) and (\ref{eq.not0}) imply that, for $q \ge \reg^\pi_{*}(\F(a))$,
$H^{r-1}(Y, \F(a) \otimes_{\O_Y} \pi^* \O_X(q)) \not= 0,$ and the result is proved.
\end{proof}

\begin{corollary}\label{cor.qCM}
Suppose that $X = \Proj A$ is equidimensional and the Rees algebra $\R = A[It]$ of $I$ is a Cohen-Macaulay ring. Let $\pi: \tx \rightarrow X$ be the blowing up of $X$ along the ideal sheaf of $I$, and let $\omega_\tx$ denote the canonical sheaf on the blowup $\tx$. Then, for $q \ge \reg\left((\omega_\tx)_1\right)$, we have
$$H^{\dim X}(\tx, \O_\tx(-1) \otimes_{\O_Y} \pi^* \O_X(q)) \not= 0.$$
\end{corollary}

\begin{proof} Let $D = \dim X$. Since $\R$ is Cohen-Macaulay, $\O_\tx$ is locally Cohen-Macaulay over $X$ and, by Lemma \ref{lem.lCM}, we have $a^*_\pi = -1$.
	
Observe that for any closed point $\pp \in X$, $\Spec \O_{X,\pp}$ is an open subset of $X$, and so it is of dimension $D$. It follows that $\R_\pp = A_{(\pp)}[I_{(\pp)}t]$ is a $(D+1)$-dimensional Cohen-Macaulay ring. Thus, $r_\pi(\O_\tx) = D+1$. By the same arguments as that of Theorem \ref{pro.higherimage}.(3), it can be deduced that for $q \ge \max\{\reg(\pi_* \O_\tx(-1)), \reg(R^D \pi_* \O_\tx(-1))\}$,
$$H^{D}(\tx, \O_\tx(-1) \otimes_{\O_Y} \pi^* \O_X(q)) \not= 0.$$

It remains to show that $\pi_* \O_\tx(-1) = 0$ and $R^D \pi_* \O_\tx(-1) = (\omega_\tx)_1$. Indeed, as in the proof of Theorem \ref{pro.higherimage}.(2), we have $\pi_* \O_\tx(-1) = (\O_\tx)_{-1}$. Since for every $\pp \in X$, $\R_\pp = A_{(\pp)}[I_{(\pp)}t]$ has no elements of degree $(-1)$, we get $(\O_\tx)_{-1} = 0$.

Let $K_\R$ be the canonical module of $\R$. Observe further that for any closed point $\pp \in X$, $H^{D+1}_{\R_{\pp+}}(\R_\pp)^\vee = K_{\R_\pp}$. Therefore, $\left[H^{D+1}_{\R_{\pp+}}(\R_{\pp})\right]_{-1} \cong (K_{\R_\pp})_1.$ Together with the Serre-Grothendieck correspondence, this implies that
$$R^D \pi_* \O_\tx(-1)\big|_{\Spec \O_{X,\pp}} = \widetilde{(K_{\R_\pp})_1}.$$
Since $K_{\R_\pp} = (K_\R)_\pp$ for all $\pp \in X$, we have $R^D \pi_* \O_\tx(-1) = (\omega_\tx)_1$, and the assertion is proven.
\end{proof}

\begin{corollary}\label{cor.qGor}
Assume the same hypotheses of Corollary \ref{cor.qCM}, and suppose further that the Rees algebra $\R = A[It]$ is Gorenstein. Then, for $q \ge \reg (\O_X)$, we have
$$H^{\dim X}(\tx, \O_\tx(-1) \otimes_{\O_Y} \pi^* \O_X(q)) \not= 0.$$
\end{corollary}

\begin{proof}
Since $\R$ is Gorenstein, we have $K_\R = \R(a^{\dim \R}(\R))$. It then follows from Lemma \ref{lem.aRees} that $K_\R = \R(-1)$. Thus, in this case, $(\omega_\tx)_1 = \widetilde{\R(-1)_1} = \widetilde{\R_0} = \O_X$, and the result follows from Corollary \ref{cor.qCM}.
\end{proof}

\begin{example}
	Let $A = \CC[[x,y]]$ and let $I = (x^2,xy)$. Then $\R = A[It]$ is Gorenstein (cf. \cite{TVZ}). It is also easy to see that $\reg (\O_X) = 0$ in this case. Thus, for $q \ge 0$, we have
	$$H^1(\tx, \O_\tx(-1) \otimes_{\O_\tx} \pi^* \O_X(q)) \not =0.$$
\end{example}

%%%%%%%%%%%%%%%%%%%%%%%%%%%

\section{$a^*$-invariant and regularity of powers of ideals} \label{sec.reg}

In this section, we will prove our main result on the $a^*$-invariant and regularity of powers of a homogeneous ideal. We shall begin by recalling our setup throughout this section.

\begin{setup} \label{setup2}
Let $A$ be a standard graded algebra over a ring $A_0$, and let $I \subseteq A$ be a homogeneous ideal generated by $(m+1)$ forms $f_0, \dots, f_m$ of degree $d > 0$. Let $X = \Proj A \subseteq \PP^n$ and let $\varphi: X \dasharrow \PP^m$ be the rational map defined by $[f_0:\dots:f_m]$. Let $\ix$ be (the closure of) the image of $\varphi$. Then, $\ix = \Proj A_0[f_0t, \dots, f_mt]$. Let $\tx \subseteq \PP^n \times \PP^m$ be (the closure of) the graph of $\varphi$. Let $\pi: \tx \rightarrow X$ and $\phi: \tx \rightarrow \ix$ be the natural projection maps. As discussed in Lemma \ref{lem.biproj}, $\pi: \tx \rightarrow X$ is the blowing-up morphism of $X$ along the ideal sheaf of $I$.
\end{setup}

\begin{remark} \label{rmk.grading}
Let $\R = A[It]$ be the Rees algebra of $I$. Then, $\R$ is naturally equipped with a bi-graded structure given by setting $\deg_\R(x) = (\deg_A(x), 0)$ for all $x \in A$ and $\deg_\R(f_it) = (d,1)$ for all $i = 0, \dots, m$. That is,
$\R = \bigoplus_{p,q \in \ZZ} \R_{(p,q)}$
where
$$\R_{(p,q)} = (I^q)_{p+dq}t^q.$$
Under this bi-graded structure, we can define
$$\O_\tx(p,q) = \pi^* \O_X(p) \otimes_{\O_\tx} \phi^* \O_{\ix}(q).$$

For $p,q \in \ZZ$, let $\R_{(*,q)} = \bigoplus_{p' \in \ZZ} \R_{(p',q)} \text{ and } \R_{(p,*)} = \bigoplus_{q' \in \ZZ} \R_{(p,q')}.$ Then, these give $\R$ natural $\ZZ$-graded structures as an algebra over $A = \R_{(*,0)}$ and over $B=A_0[f_0t, \dots, f_mt] = \R_{(0,*)}$, respectively. Furthermore, $\R_{(*,q)}$ is a graded $A$-module, and $\R_{(p,*)}$ is a graded $B$-module. Thus, we can construct coherent sheaves associated to $\R_{(*,q)}$ and $\R_{(p,*)}$ on $X$ and $\ix$, respectively.
\end{remark}

\begin{lemma}\label{lem.grading}
Let $\L = \widetilde{\R_{(*,1)}}\O_\tx$ and let $\M = \widetilde{\R_{(1,*)}}\O_\tx$. Then, $\L = \O_\tx(0,1)$ and $\M = \O_\tx(1,0)$ are twisting sheaves on $\tx$ when $\O_\tx$ is viewed as a sheaf of graded $\O_X$-algebras and as a sheaf of graded $\O_\ix$-algebras, respectively.
\end{lemma}

\begin{proof} The statement follows immediately from the described bi-graded and $\ZZ$-graded structures of $\R$ in Remark \ref{rmk.grading}.
\end{proof}

\begin{lemma} \label{lem.large}
	Let $A$ be a standard graded ring over a Noetherian ring $A_0$ and let $I \subseteq A$ be a homogeneous ideal generated in degree $d > 0$.
	\begin{enumerate}
		\item Let $q > a^*_\pi$. Then, for any $p \in \ZZ$, we have
		$$\pi_* \O_\tx(p,q) = \widetilde{\R_{(*,q)}}(p) \text{ and } R^j \pi_* \O_\tx(p,q) = 0 \ \forall \ j > 0.$$
		\item Let $p > a^*_\phi$. Then, for any $q \in \ZZ$, we have
		$$\phi_* \O_\tx(p,q) = \widetilde{\R_{(p,*)}}(q) \text{ and } R^j \phi_* \O_\tx(p,q) = 0 \ \forall \ j > 0.$$
	\end{enumerate}
\end{lemma}

\begin{proof} Observe that for each $\pp \in X = \Proj A$, $\O_{X,\pp} = A_{(\pp)}$ is the homogeneous localization of $A$ at $\pp$. Thus, $(\O_\tx)_\pp = A_{(\pp)}[I_{(\pp)}t]$. This implies that $[(\O_\tx)_\pp]_q = I_{(\pp)}^qt^q = (I^q)_{(\pp)} t^q$ for all $q \in \ZZ$. Therefore, $(\O_\tx)_q = \widetilde{\R_{(*,q)}}$. Therefore, (1) follows by applying Theorem \ref{pro.higherimage} to $\pi$, the twisting sheaf $\L = \O_\tx(0,1)$ on $\tx$, and the very ample sheaf $\O_X(1)$ on $X$. In a similar fashion, (2) follows by applying Theorem \ref{pro.higherimage} to $\phi$, the twisting sheaf $\M = \O_\tx(1,0)$ on $\tx$, and the very ample sheaf $\O_\ix(1)$ on $\ix$.
\end{proof}

We are now ready to prove the main result of this paper.

\begin{theorem} \label{thm.main}
	Let $A$ be a standard graded algebra over $A_0$ and let $I \subseteq A$ be a homogeneous ideal generated in degree $d > 0$. Let $\pi: \widetilde{X} \rightarrow X$ be the blowup of $X$ along the ideal sheaf of $I$ and let $\phi: \widetilde{X} \rightarrow \overline{X}$ be the natural projection onto its second coordinate.
\begin{enumerate}
\item For all $q > \max \left\{a^*_\pi, a^*\left(\R_{(a^*_\phi+1,*)}\right)\right\}$, we have
	$$a^*(I^q) \le dq + a^*_\phi.$$
\item For all $q \ge \max\left\{a^*_\pi+1, \reg\left(\widetilde{\R_{(a^*_\phi,*)}}\right), \reg^\phi_*(\O_\tx(a^*_\phi,0))\right\},$ we have
$$a^*(I^q) \ge dq+ a^*_\phi.$$
\end{enumerate}
In particular,
$$\stab_a(I) \le \max \left\{a^*_\pi+1, a^*\left(\R_{(a^*_\phi+1,*)}\right)+1, \reg\left(\widetilde{\R_{(a^*_\phi,*)}}\right), \reg^\phi_*(\O_\tx(a^*_\phi,0))\right\}.$$
\end{theorem}

\begin{proof} (1) By Lemma \ref{lem.large}, for $p > a^*_\phi$ and $q > a^*_\pi$, the following spectral sequences degenerate:
\begin{align*}
H^i(X, R^j \pi_* \O_\tx(p,q)) &\Rightarrow H^{i+j}(\tx, \O_\tx(p,q)) \\
H^i(\ix, R^j \phi_* \O_\tx(p,q)) &\Rightarrow H^{i+j}(\tx, \O_\tx(p,q)).
\end{align*}
Thus, for all $i \ge 0$, $p > a^*_\phi$ and $q > a^*_\pi$, we have
\begin{align}
	H^i(X, \widetilde{\R_{(*,q)}}(p)) = H^i(\tx, \O_\tx(p,q)) = H^i(\ix, \widetilde{\R_{(p,*)}}(q)). \label{eq.degenerate}
\end{align}

Observe that for $q > a^*(\R_{(p,*)})$, it follows from the Serre-Grothedieck correspondence that $H^i(\ix, \widetilde{\R_{(p,*)}}(q)) = 0$ for all $i > 0$ and $H^0(\ix, \widetilde{\R_{(p,*)}}(q)) = \R_{(p,q)}.$ Moreover, for any $p,q \in \ZZ$, we have $\widetilde{\R_{(*,q)}}(p) = \widetilde{I^q}(p+dq)$.
Therefore, for $p > a^*_\phi$ and $q > \max \{a^*_\pi, a^*(\R_{(p,*)})\}$, it follows from (\ref{eq.degenerate}) that $H^i(X,\widetilde{I^q}(p+qd)) = 0$ for all $i > 0$ and $H^0(X,\widetilde{I^q}(p+qd)) = \R_{(p,q)} = [I^q]_{p+qd}$. The Serre-Grothedieck correspondence then gives, for all $p > a^*_\phi$ and $q > \max\{a^*_\pi, a^*(\R_{(p,*)})\}$,
$$a^*(I^q) < p+dq.$$

Now, choose $p = a^*_\phi + 1 > a^*_\phi$, we get that, for all $q > \max\{a^*_\pi, a^*(\R_{(a^*_\phi+1,*)})\}$,
$$a^*(I^q) < a^*_\phi+1+dq.$$
That is, for all $q > \max\{a^*_\pi, a^*(\R_{(a^*_\phi+1,*)})\}$, we have
$$a^*(I^q) \le dq + a^*_\phi.$$

(2) As before, for $q > a^*_\pi$, the Leray spectral sequence
$$H^i(X, R^j \pi_* \O_\tx(p,q)) \Rightarrow H^{i+j}(\tx, \O_\tx(p,q)),$$
degenerates, and so we have $H^i(X, \widetilde{\R_{(*,q)}}(p)) = H^i(\tx, \O_\tx(p,q))$ for all $i \ge 0$. That is, for all $i \ge 0$, $q > a^*_\pi$ and $p \in \ZZ$,
$$H^i(X, \widetilde{I^q}(p+dq)) = H^i(\tx, \O_\tx(p,q)).$$
Moreover, by Theorem \ref{pro.higherimage}, we get that for
$q \ge \max\left\{\reg\left(\widetilde{\R_{(a^*_\phi,*)}}\right), \reg^\phi_*(\O_\tx(a^*_\phi,0))\right\},$
either $H^0(\tx, \O_\tx(a^*_\phi,q)) \not= H^0(\ix, \widetilde{\R_{(a^*_\phi,*)}}(q)) = \R_{(a^*_\phi,q)} = [I^q]_{a^*_\phi+dq}$ or $H^i(\tx, \O_\tx(a^*_\phi,q)) \not= 0$ for some $i > 0$. Hence, together with the Serre-Grothedieck correspondence, we now deduce that
$$a^*(I^q) \ge a^*_\phi+dq \ \forall \ q \ge \max\left\{a^*_\pi+1, \reg\left(\widetilde{\R_{(a^*_\phi,*)}}\right), \reg^\phi_*(\O_\tx(a^*_\phi,0))\right\}.$$
The last statement of the theorem is a straightforward consequence of (1) and (2).
\end{proof}

The following example illustrates that the bound for $q$ in Theorem \ref{thm.main}.(1) is sharp.

\begin{example} \label{ex.5}
Let $A = k[x,y]$ and let $I = (x^5, x^4y, xy^4, y^5)$. It can be seen that for $q \ge 3$, $I^q = (x,y)^{5q}$. Thus,
$$a^*(I^q) = \left\{ \begin{array}{ll} 6 & \text{ if } q = 1 \\ 10 & \text{ if } q = 2 \\ 5q-1 & \text{ otherwise.} \end{array} \right.$$
This shows that $a^*_\phi = -1$ and $\stab_a(I) = 3$.

Note that, as argued in Corollary \ref{cor.mprimary} below, the Rees algebra $\R = A[It]$ is locally Cohen-Macaulay over $\PP^1 = \Proj A$, and so $a^*_\pi = -1$. On the other hand, $\R_{(a^*_\phi+1,*)} = \R_{(0,*)} = k[It]$, and direct computation using Macaulay 2 \cite{M2} shows that $a^*(\R_{(0,*)}) = 2$.
\end{example}

As an immediate consequence of Theorem \ref{thm.main}, we recover a slight improvement of \cite[Proposition 6.7]{Ch2}. A large class of ideals which satisfy condition (1) of Corollary \ref{cor.6.7} is that of ideals for which the Rees algebras are Cohen-Macaulay.

\begin{corollary}\label{cor.6.7}
Let $A$ be a standard graded algebra over a ring $A_0$ and let $I \subseteq A$ be a homogeneous ideal generated in degree $d > 0$. Suppose that either one of the following conditions is satisfied:
\begin{enumerate}
\item for every $\pp \in \Proj A$, the Rees algebra $A_{(\pp)}[I_{(\pp)}t]$ is Cohen-Macaulay; or
\item for every $\pp \in \Proj A$, $\reg(A_{(\pp)}[I_{(\pp)}t]) = 0$.
\end{enumerate}
Then, for all $q > a^*\left(\R_{(a^*_\phi+1,*)}\right)$, we have
$$a^*(I^q) \le dq+a^*_\phi.$$
\end{corollary}

\begin{proof} By Lemmas \ref{lem.lCM} and \ref{lem.Prop64}, we have $a^*_\pi \le 0$. The assertion follows from Theorem \ref{thm.main}.
\end{proof}

When the $a^*$-invariant defect sequence is a non-increasing sequence, we obtain a bound for the stability index $\stab_a(I)$ of $I$ as follows.

\begin{corollary} \label{cor.defect}
	Let $A$ be a standard graded algebra over a ring $A_0$ and let $I \subseteq A$ be a homogeneous ideal generated in degree $d > 0$. Suppose that the sequence $\left\{a^*(I^q) - dq\right\}_{q \ge 1}$ is a non-increasing sequence. Then, for all $q > \max\left\{a^*_\pi, a^*\left(\R_{(a^*_\phi+1,*)}\right)\right\}$, we have
	$$a^*(I^q) = dq+ a^*_\phi.$$
\end{corollary}

\begin{proof} By \cite[Theorem 2.6]{Ha}, it is known that for all $q \gg 0$, $a^*(I^q) = dq+a^*_\phi$. Thus, since the sequence $\{a^*(I^q) - dq\}_{q \ge 1}$ is non-increasing, we must have $a^*(I^q) -dq \ge 0$ for all $q \ge 1$. The conclusion now follows from Theorem \ref{thm.main}.
\end{proof}

Corollary \ref{cor.defect}, particularly, recovers the equigenerated version of (\ref{eq.maincor}).

\begin{corollary}\label{cor.mprimary}
Let $A$ be a standard graded polynomial ring over a field $k$ and let $\mm$ be its maximal homogeneous ideal. Let $I \subseteq A$ be a homogeneous $\mm$-primary ideal generated in degree $d > 0$. Then, for all $q > a^*\big(\R_{(a^*_\phi+1,*)}\big)$, we have
$$a^*(I^q) = dq + a^*_\phi \text{ and } \reg(I^q) = dq + \reg_\phi.$$
In particular,
$$\stab_{\reg}(I) = \stab_a(I) \le  a^*\big(\R_{(a^*_\phi+1,*)}\big)+1.$$
\end{corollary}

\begin{proof} Since $I$ is a $\mm$-primary ideal and $A$ is a polynomial ring over a field, it is easy to see that $a^*(I^q) = a^1(I^q)$ and $\reg (I^q) = a^*(I^q) + 1$ for all $q \ge 1$. By \cite[Proposition 1.4]{EU}, we now know that the sequence $\{a^*(I^q) - dq\}_{q \ge 1}$ is a non-increasing sequence.
	
Furthermore, since $I$ is $\mm$-primary, for any $\pp \in \Proj A$, $I_{(\pp)} = (1)$. Thus, $\O_\tx$ is locally Cohen-Macaulay over $X$. This implies that $a^*_\pi = -1$. The statement then follows from Corollary \ref{cor.defect}.
\end{proof}

\begin{example} Let $A = k[x,y]$ and let $I = (x^7, x^6y, x^4y^3, x^3y^4, xy^6, y^7)$. It can be seen that for $q \ge 2$, $I^q = (x,y)^{7q}$. Thus,
$$\reg(I^q) = a^*(I^q) + 1 = \left\{\begin{array}{ll} 8 & \text{if } q = 1 \\ 7q & \text{if } q \ge 2. \end{array}\right.$$
This shows that $\reg_\phi = 0$, $a^*_\phi = -1$, and $\stab_{\reg}(I) = \stab_a(I) = 2$.

As before, it can be seen that the Rees algebra $\R = A[It]$ is locally Cohen-Macaulay over $\PP^1 = \Proj A$. This implies that $a^*_\pi = -1$. Direct computation using Macaulay2 \cite{M2} further gives 
$$a^*\big(\R_{(0,*)}\big) = 1, \R_{(-1,*)} = 0, \text{ and } \reg^\phi_*(\O_\tx(-1,0)) = 1.$$
This example shows that the bounds for the stability indexes of $I$ in Theorem \ref{thm.main} and Corollary \ref{cor.mprimary} are sharp.
\end{example}

%%%%%%%%%%%%%%%%%%%%%%%%%%


\begin{thebibliography}{99}


\bibitem{AB} A. Alilooee and A. Banerjee,
Powers of edge ideals of regularity three bipartite graphs.
J. Commut. Algebra \textbf{9} (2017), no. 4, 441-454.

\bibitem{BCH} A. Bagheri, M. Chardin and H.T. H\`a,
The eventual shape of Betti tables of powers of ideals.
Math. Res. Lett. \textbf{20} (2013), no. 6, 1033-1046.

\bibitem{BBH} A. Banerjee, S. Beyarslan and H.T. H\`a,
Regularity of powers of edge ideals: from local properties to global bounds.
Preprint (2018), arXiv:1805.01434.

\bibitem{B} D. Berlekamp,
Regularity defect stabilization of powers of an ideal.
Math. Res. Lett. \textbf{19} (2012), no. 1, 109-119.

\bibitem{BHT} S. Beyarslan, H.T. H\`a and T.N. Trung,
Regularity of powers of forests and cycles.
J. Algebraic Combin. \textbf{42} (2015), no. 4, 1077-1095.

\bibitem{Borna} K. Borna,
On linear resolution of powers of an ideal.
Osaka J. Math. \textbf{46} (2009), no. 4, 1047-1058.

\bibitem{BS} M.P. Brodmann and R.Y. Sharp,
Local cohomology: an algebraic introduction with geometric applications.
Cambridge Studies in Advanced Mathematics, 60. Cambridge University Press, Cambridge, 1998. xvi+416 pp.

\bibitem{Bruns} W. Bruns,
Algebras defined by powers of determinantal ideals.
J. Alg. \textbf{142} (1991), no. 1, 150-163.

\bibitem{BCV} W. Bruns, A. Conca and M. Varbaro,
Maximal minors and linear powers.
J. Reine Angew. Math. \textbf{702} (2015), 41-53.

\bibitem{Ch1} M. Chardin,
Powers of ideals and the cohomology of stalks and fibers of morphisms.
Algebra Number Theory \textbf{7} (2013), no. 1, 1-18.

\bibitem{Ch2} M. Chardin,
Powers of ideals: Betti numbers, cohomology and regularity.
Commutative algebra, 317-333, Springer, New York, 2013.

\bibitem{Ch3} M. Chardin,
Regularity stabilization for the powers of graded $\mm$-primary ideals.
Proc. Amer. Math. Soc. \textbf{143} (2015), no. 8, 3343-3349.

\bibitem{CHT} S.D. Cutkosky, J. Herzog, and N.V. Trung,
Asymptotic behaviour of the Castelnuovo-Mumford regularity.
Composito Mathematica, \textbf{118} (1999), 243-261.

\bibitem{EH} D. Eisenbud and J. Harris,
Powers of ideals and fibers of morphisms.
Math. Res. Lett. \textbf{17} (2010), no. 2, 267-273.

\bibitem{EHu} D. Eisenbud and C. Huneke,
Cohen-Macaulay Rees algebras and their specialization.
J. Alg. \textbf{81} (1983), 202-224.

\bibitem{EU} D. Eisenbud and B. Ulrich,
Notes on regularity stabilization.
Proc. Amer. Math. Soc. \textbf{140} (2012), no. 4, 1221-1232.

\bibitem{M2} D. R. Grayson and M. E. Stillman,
Macaulay2, a software system for research in algebraic geometry.
Available at \url{https://faculty.math.illinois.edu/Macaulay2/}

\bibitem{Gu} Y. Gu,
Regularity of powers of edge ideals of some graphs.
Acta Math. Vietnam. \textbf{42} (2017), no. 3, 445-454.

\bibitem{Ha} H.T. H\`a,
Asymptotic linearity of regularity and $a^*$-invariant of powers of ideals.
Math. Res. Lett. \textbf{18} (2011), no. 1, 1-9.

\bibitem{HT} H.T. H\`a and N.V. Trung,
Asymptotic behaviour of arithmetically Cohen-Macaulay blow-ups.
Trans. Amer. Math. Soc. \textbf{357} (2005), no. 9, 3655-3672.

\bibitem{HTT} H.T. H\`a, N.V. Trung and T.N. Trung,
Depth and regularity of powers of sums of ideals.
Math. Z. \textbf{282} (2016), no. 3-4, 819-838.

\bibitem{HangT} N.T. Hang and T.N. Trung,
Regularity of powers of cover ideals of unimodular hypergraphs.
Preprint (2017), arXiv:1705.06426.

\bibitem{Hartshorne} R. Hartshorne,
Algebraic geometry. Graduate Texts in Mathematics, No. 52.
Springer-Verlag, New York-Heidelberg, 1977. xvi+496 pp.

\bibitem{Huneke} C. Huneke,
The theory of $d$-sequences and powers of ideals.
Adv. in Math. \textbf{46} (1982), 249-279.

\bibitem{JNS} A.V. Jayanthan, N. Narayanan and S. Selvaraja,
Regularity of powers of bipartite graphs.
J. Algebraic Combin. \textbf{47} (2018), no. 1, 17-38.

\bibitem{JS1} A.V. Jayanthan and S. Selvaraja,
Asymptotic behavior of Castelnuovo-Mumford regularity of edge ideals of very well-covered graphs.
Preprint (2017), arXiv:1708.06883.

\bibitem{JS2} A.V. Jayanthan and S. Selvaraja,
Upper bounds for the regularity of powers of edge ideals of graphs.
Preprint (2018), arXiv:1805.01412.

\bibitem{K} V. Kodiyalam,
Asymptotic behaviour of Castelnuovo-Mumford regularity.
Proceedings of the American Mathematical Society, \textbf{128}, no. 2, (1999), 407-411.

\bibitem{Lu} D. Lu,
Geometric regularity of powers of two-dimensional squarefree monomial ideals.
Preprint (2018), arXiv:1808.07266.

\bibitem{Mum} D. Mumford,
Lectures on curves on an algebraic surface. With a section by G. M. Bergman.
Annals of Mathematics Studies, No. 59 Princeton University Press, Princeton, N.J. 1966 xi+200 pp.

\bibitem{R} C. Raicu,
Regularity and cohomology of determinantal thickenings.
Proc. Lond. Math. Soc. (3) \textbf{116} (2018), no. 2, 248-280.

\bibitem{Tim} T. R\"omer,
Homological properties of bigraded algebras.
Illinois J. Math. \textbf{45} (2001), no. 4, 1361-1376.

\bibitem{SF} S.A. Seyed Fakhari,
Symbolic powers of cover ideal of very well-covered and bipartite graphs.
Proc. Amer. Math. Soc. \textbf{146} (2018), no. 1, 97-110.

\bibitem{SFY} S.A. Seyed Fakhari and S. Yassemi,
Improved bounds for the regularity of powers of edge ideals of graphs.
Preprint (2018), arXiv:1805.12508.

\bibitem{Tr2} N.V. Trung,
The Castelnuovo regularity of the Rees algebra and the associated graded ring.
Trans. Amer. Math. Soc. \textbf{350} (1998), no. 7, 2813-2832.

\bibitem{TVZ} N.V. Trung, D.Q. Viet and S. Zarzuela,
When is the Rees algebra Gorenstein?
J. Alg. \textbf{175} (1995), 137-156.

\bibitem{TW} N.V. Trung and H. Wang,
On the asymptotic behavior of Castelnuovo-Mumford regularity.
J. Pure Appl. Algebra, \textbf{201} (2005), no. 1-3, 42-48.

\bibitem{W} G. Whieldon,
Stabilization of Betti tables.
J. Commut. Algebra \textbf{6} (2014), no. 1, 113-126.
	
\end{thebibliography}
\end{document}